\documentclass{article}

\usepackage{amsmath,amssymb,amsthm}
\usepackage{graphicx,caption,subcaption}
\usepackage{authblk}
\usepackage{epstopdf}

\newtheorem{theorem}{Theorem}[section]

\newtheorem{corollary}[theorem] {Corollary}
\newtheorem{definition}[theorem]{Definition}

\newtheorem{lemma} [theorem]{Lemma}

\newtheorem{proposition}[theorem]{Proposition}
\newtheorem{remark}[theorem]{Remark}

\newtheorem{notation}[theorem]{Notation}

\begin{document}

\markboth{N. K. Sudev and K. A. Germina }
{Some New Results on Strong Integer Additive Set-Indexers of Graphs}

\title{\bf Some New Results on Strong Integer Additive Set-Indexers of Graphs}

\author{N. K. SUDEV}

\affil{Department of Mathematics,\\ Vidya Academy of Science \& Technology,\\Thalakkottukara, Thrissur-680501, Kerala, India.\\
E-mail: sudevnk@gmail.com}

\author{K. A. GERMINA}

\affil{PG \& Research Department of Mathematics,\\ Mary Matha Arts \& Science College,\\ Manathavady, Wayanad-670645, Kerala, India.\\ 
E-mail: srgerminaka@gmail.com}

\maketitle

\begin{abstract}
Let $\mathbb{N}_0$ be the set of all non-negative integers. An integer additive set-indexer of a graph $G$ is an injective function $f:V(G)\to 2^{\mathbb{N}_0}$ such that the induced function $g_f:E(G) \rightarrow 2^{\mathbb{N}_0}$ defined by $f^+(uv) = f(u)+ f(v)$ is also injective. An IASI is said to be {\em $k$-uniform} if $|f^+(e)| = k$ for all $e\in E(G)$. In this paper, we introduce the notions of strong integer additive set-indexers and initiate a study of the graphs which admit strong integer additive set-indexers.
\end{abstract}

\noindent {Keywords: Set-indexers; integer additive set-indexers; strong integer additive set-sndexers; strongly uniform integer additive set-indexers.}

\noindent {Mathematics Subject Classification 2010: 05C78}

\section{Introduction}

For all  terms and definitions, not defined specifically in this paper, we refer to \cite{FH}. Unless mentioned otherwise, all graphs considered here are simple, finite and have no isolated vertices.

The researches in the field of graph labeling commenced with the introduction of the notion of graceful labeling of graphs in \cite{AR}. Analogous to the number valuations of graphs, the notion of set-valuations of graphs was introduced in \cite{A1}. 

Let us first consider the following notions of a set-labeling and a set-indexer of a given graph $G$.

\begin{definition}\textup{
\cite{A1} Let $G$ be a graph and let $X$, $Y$, $Z$ be non-empty sets. Then, the functions $f:V(G)\to 2^X$, $f:E(G)\to 2^Y$ and $f:V(G)\cup E(G)\to 2^Z$ are called the {\em  set-assignments} of vertices, edges and elements of $G$ respectively. By a {\em set-assignment} of a graph, we mean any one of them.  A set-assignment $f$ of a graph $G$ is called a {\em set-labeling} or {\em a set-valuation} of $G$ if it is injective. A graph $G$ with a set-labeling $f$ is denoted by $(G,f)$ and is called a {\em set-labeled graph} or a {\em set-valued graph}.}
\end{definition}

\begin{definition}\label{D0}\textup{
\cite{A1} For a graph $G(V,E)$ and a non-empty set $X$ of cardinality $n$, a {\em set-indexer} of $G$ is defined as an injective set-valued function $f:V(G) \to 2^{X}$ such that the function $f^{\oplus}:E(G)\to 2^{X}-\{\emptyset\}$ defined by $f^{\oplus}(uv) = f(u ){\oplus}f(v)$ for every $uv{\in} E(G)$ is also injective, where $2^{X}$ is the set of all subsets of $X$ and $\oplus$ is the symmetric difference of sets.}
\end{definition}

\begin{definition}\label{D1}\textup{
\cite{MBN} The {\em sum set} of two sets $A, B$, denoted by  $A+B$, is defined by $A + B = \{a+b: a \in A, b \in B\}$}. 
\end{definition}

Note that if either $A$ or $B$ is countably infinite, then their sum set will also be countably infinite. Hence, all sets we consider here are finite sets. We denote the cardinality of a set $A$ by $|A|$. Using the concepts of sum sets of finite sets, the notion of an integer additive set-indexer of a given graph $G$ has been introduced as follows.

\begin{definition}\label{D2}\textup{
\cite{GA} Let $\mathbb{N}_0$ be the set of all non-negative integers and let $2^{\mathbb{N}_0}$ be its power set. Then, an {\em integer additive set-indexer} (IASI, in short) of a given graph $G$ is defined as an injective function $f:V(G)\to 2^{\mathbb{N}_0}$ such that the induced function $f^+:E(G) \to 2^{\mathbb{N}_0}$ defined by $f^+ (uv) = f(u)+ f(v)$ is also injective}.
\end{definition}

\begin{definition}\label{D3}\textup{
\cite{GS1} The cardinality of the labeling set of an element (vertex or edge) of a graph $G$ is called the {\em set-indexing number} of that element.}
\end{definition}

\begin{definition}\label{DU}
\cite{GA} An IASI is said to be {\em $k$-uniform} if $|f^+(e)| = k$ for all $e\in E(G)$. That is, a connected graph $G$ is said to have a $k$-uniform IASI if all of its edges have the same set-indexing number $k$.
\end{definition}

In particular, we say that a graph $G$ has an {\em arbitrarily $k$-uniform IASI} if $G$ has a $k$-uniform IASI  for every positive integer $k$.

In \cite{GS1}, we have proved the following results on weakly uniform and arbitrarily uniform IASIs.

\begin{lemma}\label{L1}
\cite{GS1} Let $A$ and $B$ be two non-empty finite subsets of $\mathbb{N}_0$. Then $\max(|A|,|B|)\le |A+B|\le |A||B|$.
\end{lemma}

\begin{remark}\label{R1}\textup{
Due to Lemma \ref{L1}, it is clear that for an integer additive set-indexer $f$ of a graph $G$,  $max(|f(u)|, |f(v)|) \le |f^+(uv)|= |f(u)+f(v)| \le |f(u)| |f(v)|$, where $u,v \in V(G)$.}
\end{remark}

\begin{definition}\textup{
\cite{GS1} An IASI $f$ of a graph $G$ is said to be a {\em weak IASI} if $|f^+(uv)|=\max(|f(u)|,|f(v)|)$ for all $u,v\in V(G)$. }
\end{definition}

The following results are a necessary and sufficient condition for an IASI to be a weak IASI and a weakly uniform IASI respectively, of a given graph $G$.

\begin{lemma}\label{L2a}
\cite{GS1} An IASI $f$ of a graph $G$ is a weak IASI if and only if at least one end vertex of every edge of $G$ has the set-indexing number $1$.
\end{lemma}

\begin{lemma}\label{L2b}
\cite{GS1} An IASI $f$ of a graph $G$ is a weakly $k$-uniform IASI if and only if one end vertex of every edge of $G$ has the set-indexing number $1$ and the other end vertex has the set-indexing number $k$, where $k$ is a positive integer.
\end{lemma}

In this paper, we study a special type of integer additive set-indexer of a graph $G$ called a strong integer additive set-indexer of $G$.

\section{Strong Integer Additive Set Indexers}

First note that all the sets we consider during this discussion, are non-empty, finite sets of non-negative integers. In view of Lemma \ref{L1}, the maximum cardinality of the sum set of two non-empty finite sets is the product of the cardinalities of each individual sets. That is, for two non-empty finite sets $A$ and $B$, we have $|A+B|\le |A|\,|B|$. In this context, the study regarding characteristics of graphs in which the set-indexing number of every edge is the product of the set-indexing numbers of its end vertices  arises much interest. Hence, we define

\begin{definition}\textup{
If a graph $G$ has a set-indexer $f$ such that $|f^+(uv)|=|f(u)+f(v)|=|f(u)|\,|f(v)|$ for all vertices $u$ and $v$ of $G$, then $f$ is said to be a {\em strong IASI} of $G$.} 
\end{definition}

\begin{definition}\label{D6}\textup{
If all the vertices of a graph $G$ has the same set-indexing number, say $l$, then $V(G)$ is said to be {\em $l$-uniformly set-indexed}. If $G$ is a graph which admits a $k$-uniform IASI and $V(G)$ is $l$-uniformly set-indexed, then $G$ is said to have a {\em $(k,l)$-completely uniform IASI} or simply a {\em completely uniform IASI}.}
\end{definition}

\begin{lemma}\label{T2}
Let $A$, $B$ be two non-empty distinct subsets of $\mathbb{N}_0$. Then, $|A+B|=|A|\,|B|$ if and only if the differences between any two elements of one set are not equal to any differences between any two elements of the other.
\end{lemma}
\begin{proof}
Due to Lemma \ref{L1}, we have $|A+B| \le |A|\,|B|$. We have $|A+B| \ne |A||B|$ if and only if $a_i+b_j=a_r+b_s$ for some $a_i, a_r \in A$ and $b_j,b_s \in B$. That is, when $a_i- a_r = b_s-b_j$. Therefore, $|A+B|=|A|\,|B|$ if and only if $a_i- a_r \ne b_s-b_j$ for any $a_i, a_r \in A$ and $b_j,b_s \in B$. Let $D_A = \{d_{ij}: d_{ij} = a_i-a_j; a_i,a_j \in A\}$ and let $D_B = \{d_{rs}: d_{rs} = b_r-b_s; b_r,b_s \in B\}$. Then, it is clear that $|A+B|= |A|\,|B|$ if and only if no element of $D_A$ is equal to an element of $D_B$.
\end{proof}

The set of all positive differences between any two elements of a set $A$ is called the {\em difference set} of $A$ and can be denoted by $D_A$. That is, $D_A=\{a_i-a_j:a_i,a_j\in A\}$. Hence, we propose the following theorem which establishes a necessary and sufficient condition for a complete graph $K_n$ to admit a strong IASI.

\begin{notation}
Let $A$ and $B$ be two non-empty subsets of $\mathbb{N}_0$. We use the notation $A<B$ in the sense that $a\ne b$ for all $a\in A$ and $b\in B$. That is, $A<B \implies A\cap B=\emptyset$. Also, by the sequence $A_1<A_2<A_3,\ldots A_n$, we mean that all the sets $A_i$ are pairwise disjoint. 
\end{notation}

\begin{theorem}\label{TSK}
Let each vertex $v_i$ of the  complete graph $K_n$ be labeled by the set $A_i\in 2^{\mathbb{N}_0}$. Then, $K_n$ admits a strong IASI if and only if there exists a finite sequence of sets $D_1<D_2<D_3<\ldots <D_n$ where each $D_i$ is the difference set of the set-label $A_i$.
\end{theorem}
\begin{proof}
Assume that $K_n$ admits a strong IASI. Let $\{v_1,v_2,v_3,\ldots, v_n\}$ be the vertex set of $K_n$. Let each vertex $v_i$ is labeled by the set $A_i$ in $2^{\mathbb{N}_0}$. Define the set $D_i$ by $D_i = \{a_{i_r}-a_{i_s}:a_{i_r},a_{i_s}\in A_i\}$ which is the set of all differences between any two elements of $A_i$. Since the vertex $v_1$ is adjacent to $v_2$, without loss of generality, assign the set $A_1$ to $v_1$ and the set $A_2$ to the vertex $v_2$. Then, by Lemma \ref {T2}, we have $D_1<D_2$. Since $v_2$ is adjacent to $v_3$, assign $A_3$ to $v_3$ so that $D_2<D_3$. Combining these two, we get $D_1<D_2<D_3$. Proceeding in this way, for $i<j$, assign the set $A_i$ to $v_i$ and the set $A_j$ to $v_j$ so that $D_i<D_j$. After all possible assignments, we get $D_1<D_2<D_3<\ldots <D_n$.

Conversely, assume that each vertex $v_i$ of the complete graph $K_n$ is labeled by a non-empty set $A_i\in 2^{\mathbb{N}_0}$ such that there exists a finite sequence of sets $D_1<D_2<D_3<\ldots <D_n$ where each $D_i$ is the set of all differences between any two elements of the set $A_i$. Then, for each edge $v_iv_j, i<j$ we have $D_i<D_j$. Hence, by Lemma \ref{T2}, $|f^+(v_iv_j)|=|f(v_i)|\,|f(v_j)|$ for all $v_i,v_j \in V(G), i<j$. Therefore, $G$ admits a strong IASI.
\end{proof}

\begin{theorem}\label{TSS}
If a graph $G$ admits a strong IASI, then its subgraphs also admit strong IASI. 
\end{theorem}
\begin{proof}\textup{
Let $G$ be a graph which admits a strong IASI and $H$ be a subgraph of $G$. Let $g$ be the restriction of $f$ to $V(H)$. Then ${g^+}$ is the corresponding restriction of $f^+$ to $E(H)$. Then, clearly $g$ is a set-indexer on $H$. This set-indexer may be called an {\em induced set-indexer} on $H$. Since $|f^+(uv)| = |f(u)|\,|f(v)|$ for all $u,v \in V(G)$, we have $g^+(uv) =|{g}(u)|\,|{g}(v)|$ for all $u,v \in V(H)$. Hence, $H$ has a strong IASI.}
\end{proof}

The following result is a general condition for a general graph to admit a strong IASI.

\begin{corollary}\label{TSC}
A connected graph $G$ (on $n$ vertices) admits a strong IASI if and only if each vertex $v_i$ of $G$ is labeled by a set $A_i$ in $2^{\mathbb{N}_0}$ and there exists a finite sequence of sets $D_1<D_2<D_3< \ldots <D_m$, where $m\le n$ is a positive integer and each $D_i$ is the set of all differences between any two elements of $A_i$. 
\end{corollary}
\begin{proof}\textup{
Note that every graph on $n$ vertices is a subgraph of $K_n$. Then, by Theorem \ref{TSS}, a strong IASI defined on $K_n$ induces a strong IASI on $G$. Therefore, a finite subsequence of the sequence of difference sets of $K_n$ will be the sequence of difference sets of $G$. This complete the proof.}
\end{proof}

In view of Theorem \ref{TSS}, an interesting question we need to answer is whether a strong IASI of a graph $G$ induces to some other graph structures formed from $G$ by removing or identifying some of its elements (vertices or edges). First, recall the following notion. 
 
\begin{definition}\textup{
\cite{KDJ} Let $G$ be a connected graph and let $v$ be a vertex of $G$ with $d(v)=2$. Then, $v$ is adjacent to two vertices $u$ and $w$ in $G$. If $u$ and $w$ are non-adjacent vertices in $G$, then delete $v$ from $G$ and add the edge $uw$ to $G-\{v\}$. This operation is known as an {\em elementary topological reduction} on $G$. If $H$ is a graph obtained from a graph $G$ after a finite number of elementary topological reductions, then we say that $G$ and $H$ are {\em homeomorphic} graphs.}
\end{definition}

The following theorem establishes a necessary and sufficient condition for a graph homeomorphic to a given strong IASI graph $G$ to admit a strong IASI.

\begin{theorem}
Let $G$ be a graph which admits a strong IASI. Then, any graph $H$, obtained by applying a finite number of elementary topological reductions on $G$, admits a strong IASI if and only if there exist the same number of paths $P_2$, not the part of any triangle in $G$, such that the difference sets of the set-labels of their vertices are pairwise disjoint. 
\end{theorem}
\begin{proof}
Let $G$ be a graph which admits a strong IASI. Let $v$ be a vertex of $G$ with $d(v)=2$. Then $v$ is adjacent two vertices $u$ and $w$ in $G$. Then we have $|f^+(uv)|=|f(u)|\,|f(v)|$ and $|f^+(vw)|=|f(v)|\,|f(w)|$. Let $A_u$, $A_v$ and $A_w$ be the labeling sets of $u,v,w$ respectively. Also let $D_u, D_v, D_w$ be the corresponding sets of all differences between the elements of $A_u,A_v,A_w$ respectively. 

Assume that $G'=(G-v)\cup\{uw\}$ admits a strong IASI. Then, $|f^+(uw)|=|f(u)|\,|f(w)|$. Therefore, by \ref{T2}, the sets all differences of the set-labels of $u$ and $w$ are pairwise disjoint in $G'$. That is, $D_u<D_w$ in $G'$. Hence, the path $uvw$ in $G$ whose vertices have the set-labels with the difference sets of these vertices are pairwise disjoint.

Conversely, assume that the difference sets of the set-labels of whose vertices are pairwise disjoint.That is, $D_u<D_v<D_w$. Now, delete $v$ from $G$. Let $G'=(G-v)\cup \{uw\}$. Here, we have $D_u<D_w$. Then, by Lemma \ref{T2}, $|f^+(uw)|=|f(u)|\,|f(w)|$ in $G'$. All other vertices and edges are the same in $G'$ and $G$. Therefore, $G'$ admits a strong IASI. This completes the proof.
\end{proof}

Figure \ref{GHom} depicts the existence of strong IASIs of the graphs $H_1$, and $H_2$ which are obtained by applying finite number of elementary topological reductions on $G$.

\begin{figure}[h!]
	\centering
	\includegraphics[width=0.5\linewidth]{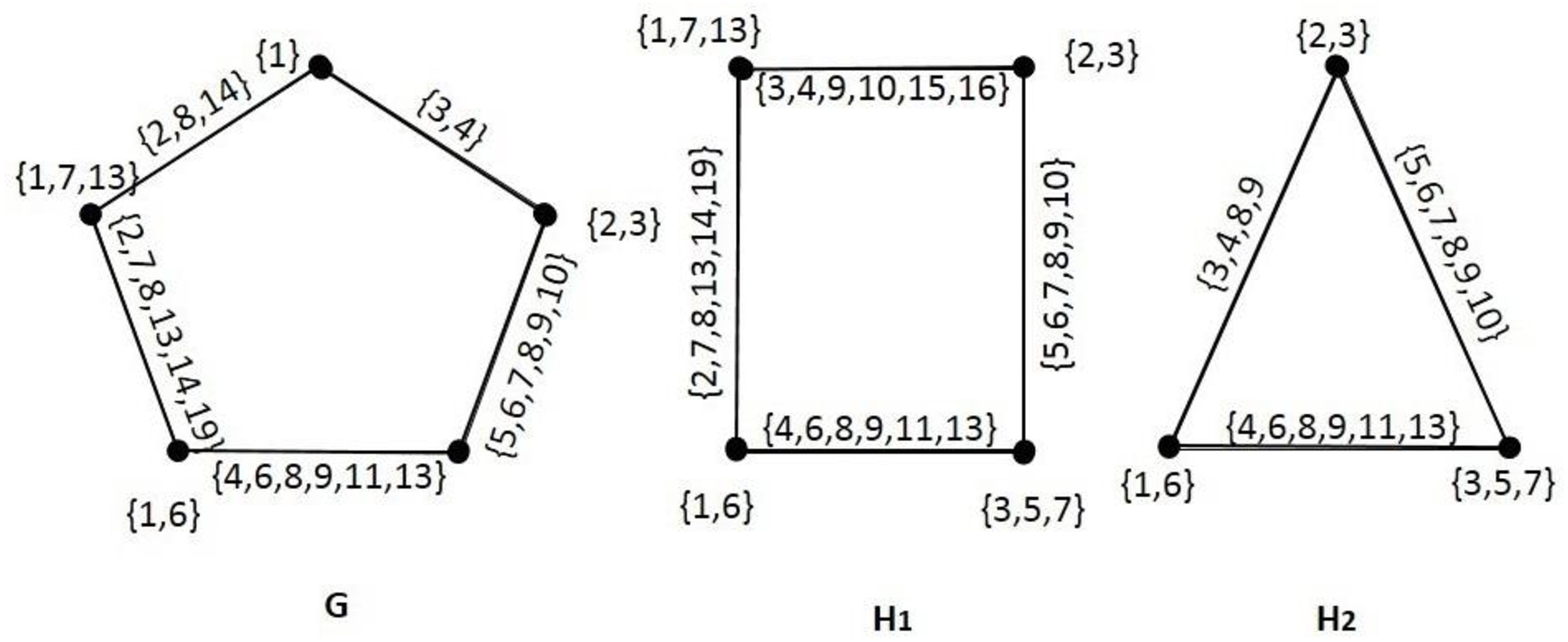}
	\caption{}
	\label{fig:GHom}
\end{figure}

In the following section, we are going to discuss about the strongly uniform integer additive set-indexers of given graphs. 

\section{Strongly Uniform Integer Additive Set-Indexers}

\begin{definition}\textup{
A graph $G$ is said to have a strongly uniform IASI if there exists an IASI $f$ defined on $G$ such that $|f^+(uv)|=k=|f(u)|\,|f(v)|$ for all $u,v \in V(G)$, where $k$ is a positive integer. }
\end{definition}

In other words, we can say that a strong integer additive set-indexer $f$ of a graph $G$ is said to be a {\em strongly $k$-uniform integer additive set-indexer} if $|f^+(e)|= k$ for all $e\in E(G)$, where $k$ is a positive integer.  In view of the above definition, we note the following.

\begin{remark}\label{R6}\textup{
If a graph $G$ admits a strongly $k$-uniform IASI, then every edge of $G$ has the set-indexing number which is the product of the set-indexing numbers of its end vertices. Equivalently, if $G$ admits a strongly $k$-uniform IASI, the set-indexing number of a vertex of $G$ is a divisor of the set-indexing number of an edge of $G$ which incidents on it.}
\end{remark}

\begin{remark} \label{R7}\textup{
Let $G$ be a graph which admits a strongly $k$-uniform IASI. Let $n$ be the number of divisors of $k$. Then by Remark \ref{R6},  each vertex of $G$  has some set-indexing number $d_i$, which is a divisor of $k$. Hence, $V(G)$ can be partitioned into at most $n$ sets, say $(X_1, X_2, \ldots, X_n)$ where each $X_i$ consists of the vertices of $G$ having the set-indexing number $d_i$.}
\end{remark}

Let us now proceed to discuss the admissibility of strongly uniform IASIs by certain graphs and conditions required for their existence for various graphs. The following theorem establishes the admissibility of strongly uniform IASIs by bipartite graphs.

\begin{theorem}\label{TBS}
For any positive integer $k$, all bipartite graphs admit a strongly $k$-uniform IASI.
\end{theorem}
\begin{proof}\textup{
Let $G$ be a bipartite graph. Let $(X, Y)$ be the bipartition of $V(G)$. Our aim is to develop a set-indexer for $G$ so that each edge of it has a set-indexing number $k$. We label the vertices of $X$ and $Y$ in the following way.  Let $m$ and $n$ be two positive integers such that $mn=k$. Now assign the set $\{i,i+1,i+2, \cdots i+(m-1)\}$ to the vertex $u_i$ in $X$. Let $v_j$ be a vertex in $Y$ which is adjacent to $u_i$. Now assign the set $\{j, j+m, j+2m,\cdots j+(n-1)m\}$ to the vertex $v_j$. Now the edge $u_iv_j$ has the set-indexer $\{i+j, i+j+m, i+j+2m, \cdots i+j+(n-1)m, i+1+j, i+1+j+m, i+1+j+2m, \cdots i+1+j+(n-1)m, \cdots i+(m-1)+j, i+(m-1)+j+m, \cdots i+(m-1)+j+(n-1)m\}$. That is,  the set-indexer of $u_iv_j$ is $\{i+j,i+j+1,i+j+2, \cdots, i+j+(m-1), i+j+m, i+j+m+1, i+j+m+2, \cdots i+j+2m-1, i+j+2m,i+j+2m+1, \cdots, 1+j+(n-1)m, \cdots, i+j+mn-1\}$.
Therefore, the set-indexing number of $u_iv_j$ is $mn=k$. Since $u_i$ and $v_j$ are arbitrary elements of $X$ and $Y$ respectively, the above argument can be extended to all edges in $G$. That is, all edges of $G$ can be assigned by a $k$-element set. Therefore, $G$ admits a strongly $k$-uniform IASI.}
\end{proof}

The following result is a special case when the number $k$, we consider above is a prime number.

\begin{corollary}\label{CBP}
If $p$ is a prime number, then a strongly $p$-uniform IASI of a bipartite graph $G$ is also a weakly $p$-uniform IASI for $G$.
\end{corollary}
\begin{proof}\textup{
Let $f$ be a strongly $p$-uniform IASI of a given graph $G$, where $p$ be a prime integer. Then, $p$ has only two divisors $1$ and $p$. By Theorem \ref{TBS}, the given graph $G$ has a strongly $p$-uniform IASI if $m=1$ and $n=p$ (or $m=p$ and $n=1$). Hence, by Lemma \ref{L2b}, this IASI $f$ is also a weakly $p$-uniform IASI of $G$.}
\end{proof}

Invoking Theorem \ref{TBS} and Corollary \ref{CBP}, we have the following result.

\begin{corollary}\label{CBC}
A bipartite graph has distinct weakly $k$-uniform IASI and strongly $k$-uniform IASI if and only if $k$ is a positive composite integer.
\end{corollary}
\begin{proof}\textup{
By Theorem \ref{TBS}, it is clear that every bipartite graph admits a strongly $k$-uniform IASI. If $k$ is a composite number, we can find two integers $m$ and $n$ such that $m \ne n\ne 1$ and $mn=k$. Then by Theorem \ref{TBS}, the given graph $G$ admits a strongly $k$-uniform IASI where $k=mn$. Conversely, let $G$ has distinct weakly and strongly $k$-uniform IASIs. Then by Corollary \ref{CBP}, $p$ is not a prime number. Therefore, $k$ is a positive composite integer.}
\end{proof}

Figure \ref{SU-Graphs} illustrates some strong uniform IASIs for bipartite graphs. It may be noted that the second graph in Figure \ref{SU-Graphs} is a 4-uniform IASI as well as a $(4,2)$-completely uniform IASI.

\begin{figure}[h!]
\centering
\includegraphics[width=0.5\linewidth]{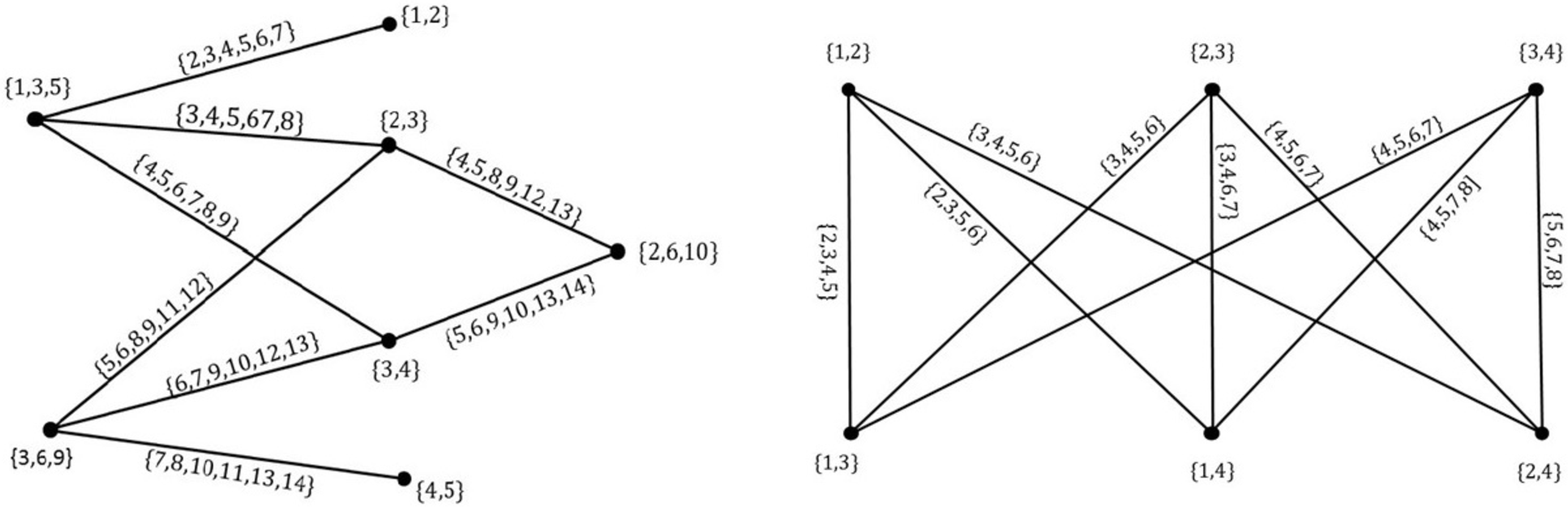}
\caption{}
\label{fig:SU-Graphs}
\end{figure}

The following theorem discusses a condition required for a complete graph to admit a strongly $k$-uniform IASI.

\begin{theorem}\label{TCG}
A strongly $k$- uniform IASI of a complete graph $K_n$  is a $(k,l)$-completely uniform IASI, where $l=\sqrt{k}$.
\end{theorem}
\begin{proof}
Since all distinct pairs of vertices are adjacent  in $K_n$,  $K_n$ admits a strongly $k$- uniform IASI when the set-indexing number of all vertices of $K_n$ are equal, say $l$. Therefore, $K_n$ admits $(k,l)$-completely uniform IASI. Here $l^2=k$.
\end{proof}

To proceed further, we need the following results from the number theory.

\begin{proposition}\label{R8}
\cite{TMA} If $m$ a non-square integer, then  for every divisor $d$ of $m$ there exists another divisor $\frac{m}{d}$ for $m$. That is, the number of divisors of a non-square integer is even.
\end{proposition}

\begin{proposition}\label{R9}
\cite{TMA} If $m$ a perfect square integer, then  for every divisor $d$ of $m$ there exists another divisor $\frac{m}{d}$ for $m$ except for the divisor $d_i = \sqrt{k}$. That is, the number of divisors of a  perfect square is odd.
\end{proposition}

In view of Proposition \ref{R8} and Proposition \ref{R9}, we have the following theorem.

\begin{theorem}\label{M3}
Let $G$ be a graph which admits a strongly $k$-uniform IASI. Also, let $n$ be the number of divisors of $k$. Then if $k$ is a non-square integer, then $G$ has at most  $\frac{n}{2}$ bipartite components and if  $k$ is a perfect square integer, then $G$ has at most  of  $\frac{n+1}{2}$ components in which at most $\frac{n-1}{2}$ components are bipartite components.
\end {theorem}
\begin{proof}
Let $G$ be the graph which admits a strongly $k$-uniform IASI. Also let $k=d_1d_2d_3 ....... d_n$.

\noindent {\bf Case 1:} Let $k$ be a non-square integer. Then, by Result \ref{R8}, the number of divisors $n$ of $k$ is even. By Remark \ref{R7}, $V(G)$ can be partitioned into at most $n$ sets. Let this partition be $(X_1,X_2,X_3, ....., X_n)$, where each $X_i$ is the set of vertices having set-indexing number $d_i$. Since $G$ admits a maximal $k$-uniform IASI, each vertex of $X_i$ is adjacent to some vertices of $X_j$ so that $d_i.d_j=k$. Also, no two vertices of $X_i$ and no two vertices of $X_j$ are adjacent and no two vertices of both the sets can not be adjacent to the vertices of any other set $X_k$ of $G$. That is, $(X_i,X_j)$ is a component of $G$ which is bipartite. We can find at most $\frac{n}{2}$ such bipartite components $(X_i, X_j)$ of $G$.

\noindent{\bf Case 2:} Let $k$ be a perfect square. Then, by Result \ref{R9}, the number of divisors $n$ of $k$ is odd. As in Case 1, for each $X_i$ of $V(G)$, we can find a set $X_j$   such that $(X_i, X_j)$ is a  bipartite component of $G$ except for the set $X_r$ whose corresponding divisor $d_r=\sqrt{k}$. But no vertex of $X_r$ can be adjacent to a vertex of any other set $X_i, i\ne r$. Hence, the induced subgraph $[X_r]$ of $G$ is a component of $G$ such that the vertices in $X_r$ are adjacent to some vertices of itself. Therefore, $G$ has at most  $\frac{n-1}{2}$ bipartite components and the component $[X_r]$. That is, $G$ has at most $\frac{n+1}{2}$ components of which $\frac{n-1}{2}$ are bipartite components.
\end{proof}

\begin{remark}\textup{
From Theorem \ref{M3}, if the vertex set of a graph $G$ admitting a strongly $k$-uniform IASI, is partitioned into more than two sets, then $G$ is a disconnected graph.}
\end{remark}

\begin{corollary}\label{CSB}
Let $k$ be a non-square integer. Then a graph $G$ admits a strongly k-uniform IASI if and only if $G$ is bipartite or a union of disjoint bipartite components.
\end{corollary}
\begin{proof}\textup{
Let $k$ be a non-square integer and let the graph admits a strongly $k$-uniform IASI. Then by Remark \ref{R7}, $V(G)$ can be partitioned into at most $n$ sets, where $n$ is even, and $G$ has at most $\frac{n}{2}$ bipartite components. If $n=2$, $G$ has one bipartite component and hence $G$ is a bipartite graph. If $n>2$, then by Theorem \ref{CSB}, $G$ has more than one bipartite component. Hence, $G$ is a union of bipartite components.}
\end{proof}

Recall that a {\em clique} of a graph $G$  is a maximal complete subgraph of $G$.

\begin{corollary}\label{M4}
If a graph $G$, which admits a strongly $k$-uniform IASI, then it contains at most one component which is a clique.
\end{corollary}
\begin{proof}\textup{
From Theorem \ref{M3}, it is clear that all components of $G$ are bipartite  if $k$ is a non-square integer. If $k$ is a perfect square, then there exists exactly one component of $G$  in which some vertices are adjacent to some other vertices of itself. Hence, $G$ can contain at most one component that is clique.}
\end{proof}

\begin{corollary}
Let the graph $G$ has a strongly $k$-uniform IASI. Hence, if $G$ has a component which is a clique, then $k$ is a perfect square.
\end{corollary}
\begin{proof}\textup{
Let the graph $G$ admits a strongly $k$-uniform IASI. From Theorem \ref{M3}, if $k$ is a non-square integer, all the components of $G$ are bipartite. Then, if one component of $G$ is a clique, it can not be bipartite. Therefore, $k$  can not be a non-square integer. Hence, $k$ is a perfect square.}
\end{proof}

In view of the results discussed above, we establish the following theorem.

\begin{theorem}\label{TNB}
A connected non-bipartite graph $G$ admits a strongly $k$-uniform IASI if and only if $k$ is a perfect square and this IASI is a $(k,l)$-completely uniform, where $l=\sqrt{k}$.
\end{theorem}
\begin{proof}\textup{
Let $G$ be a connected non-bipartite graph which admits a strongly $k$-uniform IASI. Then, by Corollary \ref{CSB}, $k$ must be a perfect square. Then, by Case 2 of Theorem \ref{M3}, since $G$ is connected, each vertex of $G$ must have a set-indexing number $l=\sqrt{k}$. Hence, this IASI of $G$ is $(k,l))$-completely uniform IASI. Conversely, assume that $G$ is $(k,l)$-completely uniform IASI where $l=\sqrt{k}$. Then each edge of $G$ has the set-indexing number $k=l^2$ and hence $G$ is strongly $k$-uniform IASI.}
\end{proof}

Invoking the above results, more generally we have

\begin{theorem}
A connected graph $G$ admits a strongly $k$-uniform IASI if and only if $G$ is bipartite or this IASI is a $(k,l)$-completely uniform IASI of $G$, where $k=l^2$.
\end{theorem}

\section{Conclusion}
In this paper, we have given some characteristics of the graphs which admit strong IASIs and strongly $k$-uniform IASIs. More properties and characteristics of weak and strong IASIs, both uniform and non-uniform, are yet to be investigated. More studies may be done in the field of IASI when the ground set $X$ is finite instead of the set $\mathbb{N}_0$. 

We have formulated some conditions for certain graphs to admit weakly and strongly uniform IASIs. The problems of establishing the necessary and sufficient conditions for various graphs and graph classes to admit these types of uniform IASIs have still been unsettled. Certain IASIs which assign set-labels having specific properties, to the elements of given graphs also seem to be promising for further investigations. All these facts highlight the wide scope for intensive studies in this area.

\end{document}